\documentclass[a4paper,11pt]{article}

\usepackage{amsmath}
\usepackage{amssymb}
\usepackage{amsthm}
\usepackage{graphicx}
\usepackage{amscd}
\usepackage{epic, eepic}
\usepackage{url}
\usepackage{color}
\usepackage[utf8]{inputenc} 
\usepackage{comment}
\usepackage{enumerate}   
\usepackage{todonotes}
\usepackage{graphicx}
\usepackage{epstopdf}
\usepackage{enumitem}
\usepackage{tikz}
\usepackage[top=25mm, bottom=25mm, left=28mm, right=28mm]{geometry}
\usepackage[bookmarks=false,hidelinks]{hyperref}

\makeatletter
\renewenvironment{proof}[1][\proofname] {\par\pushQED{\qed}\normalfont\topsep6\p@\@plus6\p@\relax\trivlist\item[\hskip\labelsep\bfseries#1\@addpunct{.}]\ignorespaces}{\popQED\endtrivlist\@endpefalse}
\makeatother

\newtheorem{theorem}{\bf Theorem}[section]
\newtheorem{lemma}[theorem]{\bf Lemma}

\newtheorem{conjecture}[theorem]{\bf Conjecture}

\theoremstyle{definition}

\newtheorem{definition}[theorem]{\bf Definition}

\def\eps{\varepsilon}
\def\vx{\textbf{x}}
\def\vy{\textbf{y}}
\def\vz{\textbf{z}}

\def\cG{\mathcal{G}}
\def\cC{\mathcal{C}}
\def\cD{\mathcal{D}}
\def\ex{\mathrm{ex}}

\title{Disproof of a conjecture of Erd\H os and Simonovits on the Tur\'an number of graphs with minimum degree $3$}
\author{Oliver Janzer\thanks{Department of Mathematics, ETH Z\"urich, Switzerland. Email: \href{mailto:oliver.janzer@math.ethz.ch} {\nolinkurl{oliver.janzer@math.ethz.ch}}.
Research supported by an ETH Z\"urich Postdoctoral Fellowship 20-1 FEL-35.}}
\date{}

\begin{document}
	
\maketitle
	
\begin{abstract}
	In 1981, Erd\H os and Simonovits conjectured that for any bipartite graph $H$ we have $\mathrm{ex}(n,H)=O(n^{3/2})$ if and only if $H$ is $2$-degenerate. Later, Erd\H os offered 250 dollars for a proof and 500 dollars for a counterexample. In this paper, we disprove the conjecture by finding, for any $\varepsilon>0$, a $3$-regular bipartite graph $H$ with $\mathrm{ex}(n,H)=O(n^{4/3+\varepsilon})$.
\end{abstract}

\section{Introduction}

For a graph $H$ and a positive integer $n$, the Tur\'an number (or extremal number) $\ex(n,H)$ denotes the maximum possible number of edges in an $n$-vertex graph which does not contain $H$ as a subgraph. The celebrated Erd\H os--Stone--Simonovits \cite{ES46,ESi66} theorem states that $\ex(n,H)=(1-\frac{1}{\chi(H)-1}+o(1))\binom{n}{2}$, which determines the asymptotics of $\ex(n,H)$ whenever $H$ is not bipartite. However, finding good bounds for $\ex(n,H)$ for bipartite graphs $H$ is in general very difficult. The order of magnitude of $\ex(n,H)$ is not known even in some very basic cases such as when $H$ is the even cycle $C_{8}$, complete bipartite graph $K_{4,4}$ or the cube $Q_3$. One of the few general results is the following upper bound, due to F\"uredi~\cite{Fu91} (later reproved using different methods by Alon, Krivelevich and Sudakov \cite{AKS03}).

\begin{theorem}[F\"uredi \cite{Fu91} and Alon--Krivelevich--Sudakov \cite{AKS03}] \label{thm:bounded max degree}
	Let $H$ be a bipartite graph in which all vertices in one of the parts have degree at most $r$. Then $\ex(n,H)=O(n^{2-1/r})$.
\end{theorem}

Let us call a graph $r$-degenerate if each of its subgraphs has minimum degree at most $r$. Erd\H os conjectured that the following more general version of Theorem \ref{thm:bounded max degree} is also true.

\begin{conjecture}[Erd\H os \cite{Erd67}] \label{conj:degenerate}
	Let $H$ be an $r$-degenerate bipartite graph. Then $\ex(n,H)=O(n^{2-1/r})$.
\end{conjecture}

Although the conjecture is still wide open, Alon, Krivelevich and Sudakov \cite{AKS03} proved that if a bipartite graph $H$ is $r$-degenerate, then $\ex(n,H)=O(n^{2-\frac{1}{4r}})$.

As a kind of converse to Conjecture \ref{conj:degenerate}, Erd\H os and Simonovits \cite{Erd93,ESonline} conjectured that if $H$ has minimum degree at least $r+1$, then there exists $\eps>0$ such that $\ex(n,H)=\Omega(n^{2-1/r+\eps})$. This was recently disproved by the author \cite{janzerrainbow} for all $r\geq 3$ (the conjecture is true for $r=1$). The case $r=2$ of the conjecture had been raised much earlier \cite{Erd81new}. Erd\H os was particularly interested in this case as he stated the problem several times \cite{Erd81new,Erd83,Erd90,Erd93} and offered a prize for a proof or disproof of the following statement.

\begin{conjecture}[Erd\H os--Simonovits \cite{Erd81new,Erd83,Erd90,Erd93,ESonlinenew}, \$250 for a proof and \$500 for a counterexample] \label{conj:2degenerate}
	Let $H$ be a bipartite graph. Then $\ex(n,H)=O(n^{3/2})$ holds if and only if $H$ is $2$-degenerate.
\end{conjecture}

We disprove this conjecture in a rather strong sense.

\begin{theorem} \label{thm:disproofES}
	For any $\eps>0$ there is a $3$-regular (bipartite) graph $H$ for which $\ex(n,H)=O(n^{4/3+\eps})$.
\end{theorem}

This bound is best possible. Indeed, the probabilistic deletion method shows that for any $3$-regular graph $H$ there is some $\delta>0$ such that $\ex(n,H)=\Omega(n^{4/3+\delta})$.

We now define the graphs which we will show are counterexamples to Conjecture \ref{conj:2degenerate}.

\begin{definition}
	For positive integers $k\geq 1$ and $\ell\geq 2$, let $H_{k,\ell}$ be the following graph.
	
	Let
	$$V(H_{k,\ell})=\{x_{i,j}:1\leq i\leq 4k,1\leq j\leq 2\ell\}$$
	and let
	\begin{align*}
		E(H_{k,\ell})
		&= \{x_{2i-1,j}x_{2i,j}: 1\leq i\leq 2k,1\leq j\leq 2\ell\} \\
		&\cup \{x_{1,j}x_{1,j+1}: 1\leq j\leq 2\ell\}\cup \{x_{4k,j}x_{4k,j+1}: 1\leq j\leq 2\ell\} \\
		&\cup \{x_{2i,j}x_{2i+1,j+1}: 1\leq i\leq 2k-1,1\leq j\leq 2\ell\} \\
		&\cup \{x_{2i+1,j}x_{2i,j+1}: 1\leq i\leq 2k-1,1\leq j\leq 2\ell\},
	\end{align*}
	where $x_{i,2\ell+1}:=x_{i,1}$.
	See Figure \ref{fig:H24} for the picture of $H_{2,4}$.
\end{definition}

\begin{figure}
	\centering
	\begin{tikzpicture}[scale=0.6]
	
	\draw[fill=black](0,0)circle(4pt);
    \draw[fill=black](0,1)circle(4pt);
    \draw[fill=black](0,2)circle(4pt);
    \draw[fill=black](0,3)circle(4pt);
    \draw[fill=black](0,4)circle(4pt);
    \draw[fill=black](0,5)circle(4pt);
    \draw[fill=black](0,6)circle(4pt);
    \draw[fill=black](0,7)circle(4pt);

	\draw[fill=black](2,0)circle(4pt);
    \draw[fill=black](2,1)circle(4pt);
    \draw[fill=black](2,2)circle(4pt);
    \draw[fill=black](2,3)circle(4pt);
    \draw[fill=black](2,4)circle(4pt);
    \draw[fill=black](2,5)circle(4pt);
    \draw[fill=black](2,6)circle(4pt);
    \draw[fill=black](2,7)circle(4pt);
    
    \draw[fill=black](4,0)circle(4pt);
    \draw[fill=black](4,1)circle(4pt);
    \draw[fill=black](4,2)circle(4pt);
    \draw[fill=black](4,3)circle(4pt);
    \draw[fill=black](4,4)circle(4pt);
    \draw[fill=black](4,5)circle(4pt);
    \draw[fill=black](4,6)circle(4pt);
    \draw[fill=black](4,7)circle(4pt);
	
	\draw[fill=black](0,11)circle(4pt);
    \draw[fill=black](0,12)circle(4pt);
    \draw[fill=black](0,13)circle(4pt);
    \draw[fill=black](0,14)circle(4pt);
    \draw[fill=black](0,15)circle(4pt);
    \draw[fill=black](0,16)circle(4pt);
    \draw[fill=black](0,17)circle(4pt);
    \draw[fill=black](0,18)circle(4pt);

	\draw[fill=black](2,11)circle(4pt);
    \draw[fill=black](2,12)circle(4pt);
    \draw[fill=black](2,13)circle(4pt);
    \draw[fill=black](2,14)circle(4pt);
    \draw[fill=black](2,15)circle(4pt);
    \draw[fill=black](2,16)circle(4pt);
    \draw[fill=black](2,17)circle(4pt);
    \draw[fill=black](2,18)circle(4pt);
    
    \draw[fill=black](4,11)circle(4pt);
    \draw[fill=black](4,12)circle(4pt);
    \draw[fill=black](4,13)circle(4pt);
    \draw[fill=black](4,14)circle(4pt);
    \draw[fill=black](4,15)circle(4pt);
    \draw[fill=black](4,16)circle(4pt);
    \draw[fill=black](4,17)circle(4pt);
    \draw[fill=black](4,18)circle(4pt);
    
    \draw[fill=black](-3,9)circle(4pt);
    \draw[fill=black](-4,9)circle(4pt);
    \draw[fill=black](-5,9)circle(4pt);
    \draw[fill=black](-6,9)circle(4pt);
    \draw[fill=black](-7,9)circle(4pt);
    \draw[fill=black](-8,9)circle(4pt);
    \draw[fill=black](-9,9)circle(4pt);
    \draw[fill=black](-2,9)circle(4pt);
    
    \draw[fill=black](6,9)circle(4pt);
    \draw[fill=black](7,9)circle(4pt);
    \draw[fill=black](8,9)circle(4pt);
    \draw[fill=black](9,9)circle(4pt);
    \draw[fill=black](10,9)circle(4pt);
    \draw[fill=black](11,9)circle(4pt);
    \draw[fill=black](12,9)circle(4pt);
    \draw[fill=black](13,9)circle(4pt);

    \draw[thick](0,0)--(2,0)--(2,1)--(0,2)--(0,3)--(2,4)--(2,5)--(0,6)--(0,7)--(2,7)--(2,6)--(0,5)--(0,4)--(2,3)--(2,2)--(0,1)--(0,0);

	\draw[thick](2,0)--(4,0)--(4,1)--(2,2)--(2,3)--(4,4)--(4,5)--(2,6)--(2,7)--(4,7)--(4,6)--(2,5)--(2,4)--(4,3)--(4,2)--(2,1)--(2,0);
	
	\draw[thick](0,11)--(2,11)--(2,12)--(0,13)--(0,14)--(2,15)--(2,16)--(0,17)--(0,18)--(2,18)--(2,17)--(0,16)--(0,15)--(2,14)--(2,13)--(0,12)--(0,11);

	\draw[thick](2,11)--(4,11)--(4,12)--(2,13)--(2,14)--(4,15)--(4,16)--(2,17)--(2,18)--(4,18)--(4,17)--(2,16)--(2,15)--(4,14)--(4,13)--(2,12)--(2,11);
	
	\draw[thick](-9,9)--(-8,9)(-7,9)--(-6,9)(-5,9)--(-4,9)(-3,9)--(-2,9);
	
	\draw[thick](6,9)--(7,9)(8,9)--(9,9)(10,9)--(11,9)(12,9)--(13,9);
	
	\draw[thick](-9,9)--(0,18)(-8,9)--(0,16)(-7,9)--(0,17)(-6,9)--(0,14)(-5,9)--(0,15)(-4,9)--(0,12)(-3,9)--(0,13)(-2,9)--(0,11);
	
	\draw[thick](-9,9)--(0,0)(-8,9)--(0,2)(-7,9)--(0,1)(-6,9)--(0,4)(-5,9)--(0,3)(-4,9)--(0,6)(-3,9)--(0,5)(-2,9)--(0,7);
	
	\draw[thick](13,9)--(4,18)(12,9)--(4,16)(11,9)--(4,17)(10,9)--(4,14)(9,9)--(4,15)(8,9)--(4,12)(7,9)--(4,13)(6,9)--(4,11);
	
	\draw[thick](13,9)--(4,0)(12,9)--(4,2)(11,9)--(4,1)(10,9)--(4,4)(9,9)--(4,3)(8,9)--(4,6)(7,9)--(4,5)(6,9)--(4,7);
	
	\node at (-0.7,0) {$x_{8,1}$};
	\node at (-0.7,1) {$x_{7,1}$};
	\node at (-0.7,2) {$x_{6,1}$};
	\node at (-0.7,3) {$x_{5,1}$};
	\node at (-0.7,4) {$x_{4,1}$};
	\node at (-0.7,5) {$x_{3,1}$};
	\node at (-0.7,6) {$x_{2,1}$};
	\node at (-0.7,7) {$x_{1,1}$};
	
	\node at (1.3,0) {$x_{8,2}$};
	\node at (1.3,1) {$x_{7,2}$};
	\node at (1.3,2) {$x_{6,2}$};
	\node at (1.3,3) {$x_{5,2}$};
	\node at (1.3,4) {$x_{4,2}$};
	\node at (1.3,5) {$x_{3,2}$};
	\node at (1.3,6) {$x_{2,2}$};
	\node at (1.3,7) {$x_{1,2}$};
	
	\node at (3.3,0) {$x_{8,3}$};
	\node at (3.3,1) {$x_{7,3}$};
	\node at (3.3,2) {$x_{6,3}$};
	\node at (3.3,3) {$x_{5,3}$};
	\node at (3.3,4) {$x_{4,3}$};
	\node at (3.3,5) {$x_{3,3}$};
	\node at (3.3,6) {$x_{2,3}$};
	\node at (3.3,7) {$x_{1,3}$};
	
	\node at (6,8.5) {$x_{1,4}$};
	\node at (7,9.5) {$x_{2,4}$};
	\node at (8,8.5) {$x_{3,4}$};
	\node at (9,9.5) {$x_{4,4}$};
	\node at (10,8.5) {$x_{5,4}$};
	\node at (11,9.5) {$x_{6,4}$};
	\node at (12,8.5) {$x_{7,4}$};
	\node at (13,9.5) {$x_{8,4}$};

	\end{tikzpicture}		
	\caption{The graph $H_{2,4}$}
	\label{fig:H24}
\end{figure}

Since $H_{k,\ell}$ is $3$-regular for every $k$ and $\ell$, Theorem \ref{thm:disproofES} follows from the next result.

\begin{theorem} \label{thm:turan of Hkl}
	Let $0<\eps<1/6$. If $k\geq 1/\eps$ and $\ell\geq 16k/\eps$, then $\ex(n,H_{k,\ell})=O(n^{4/3+\eps})$.
\end{theorem}

\subsection{Overview of the proof} \label{sec:overview}

In this subsection we give an overview of the proof of Theorem \ref{thm:turan of Hkl}. Let $G$ be an $n$-vertex graph with roughly $n^{4/3+\eps}$ edges. We seek to prove that it contains $H_{k,\ell}$ as a subgraph provided that $k$ and $\ell$ are large enough. By standard regularization methods, we may assume that $G$ is almost-regular, i.e that its maximum degree is at most a constant times the minimum degree.

Let us consider the following auxiliary graph $\cG$. The vertices of $\cG$ are (labelled) matchings of size $2k$ in $G$, or in other words $4k$-tuples $(x_1,x_2,\dots,x_{4k})\in V(G)^{4k}$ of distinct vertices such that $x_{2i-1}x_{2i}\in E(G)$ for every $1\leq i\leq 2k$. We join the vertices $(x_1,x_2,\dots,x_{4k})\in V(\cG)$ and $(y_1,y_2,\dots,y_{4k})\in V(\cG)$ by an edge in $\cG$ if $$x_1x_2y_3y_4x_5x_6y_7y_8\dots x_{4k-3}x_{4k-2}y_{4k-1}y_{4k}x_{4k}x_{4k-1}y_{4k-2}y_{4k-3}\dots x_4x_3y_2y_1x_1$$ is a cycle of length $8k$ in $G$.

Now the number of vertices in $\cG$ is equal to the number of (labelled) matchings of size $2k$ in $G$ which is at most $e(G)^{2k}\approx (n^{4/3+\eps})^{2k}$. On the other hand, the number of edges in $\cG$ is roughly equal to the number of cycles of length $8k$ in $G$. Since $G$ has $n^{4/3+\eps}$ edges, which is much greater than the Tur\'an number of $C_{8k}$, it follows by standard supersaturation results that the number of $8k$-cycles in $G$ is at least roughly $(n^{1/3+\eps})^{8k}=(n^{4/3+4\eps})^{2k}\approx |V(\cG)|^{\frac{4/3+4\eps}{4/3+\eps}}\geq |V(\cG)|^{1+\eps}$. Hence, if $\ell$ is much larger than $1/\eps$, then the number of edges in $\cG$ is much larger than the Tur\'an number of $C_{2\ell}$, and so $\cG$ contains a $2\ell$-cycle. Let the vertices of such a $2\ell$-cycle be $\vx^1,\vx^2,\dots,\vx^{2\ell}$ in the natural order, where $\vx^j=(x_{1,j},x_{2,j},\dots,x_{4k,j})$. Now if all the vertices $x_{i,j}$ are distinct, then they form a copy of $H_{k,\ell}$. Thus, our goal is to find a $2\ell$-cycle in $\cG$ in which the vertices do not share coordinates. To do this, we use the machinery from \cite{janzerrainbow}.

The techniques from \cite{janzerrainbow} would directly apply if we could show that for any matching $M$ of size $2k$ and any vertex $u\in V(G)$ outside $M$, the number of $8k$-cycles extending $M$ and containing $u$ is a small proportion of the number of $8k$-cycles extending $M$. Since we cannot prove this, we restrict our attention to a certain collection $\cC$ of $8k$-cycles in $G$ and use only these $8k$-cycles to define the auxiliary graph $\cG$. That is, we join $\vx$ and $\vy$ by an edge in $\cG$ if their coordinates form an $8k$-cycle (in a prescribed order) which belongs to $\cC$. Roughly speaking, we prove that we can find a large collection $\mathcal{C}$ of $8k$-cycles in $G$ such that for any matching $M$ of size $2k$ and any vertex $u$ outside $M$, the number of elements of $\cC$ extending $M$ and containing $u$ is a small proportion of the number of elements of $\cC$ extending $M$. This will allow us to apply results from \cite{janzerrainbow} to find a $2\ell$-cycle in $\cG$ whose vertices do not share coordinates.

\section{The proof of Theorem \ref{thm:turan of Hkl}}

\textbf{Notation and terminology.} \sloppy A \emph{homomorphic $2k$-cycle} in a graph $G$ is a $2k$-tuple $(x_1,x_2,\dots,x_{2k})\in V(G)^{2k}$ with the property that $x_1x_2,x_2x_3,\dots,x_{2k}x_1\in E(G)$. Let us write $\hom(C_{2k},G)$ for the number of homomorphic $2k$-cycles in $G$. This is also equal to the number of graph homomorphisms from $C_{2k}$ to $G$. We write $\delta(G)$ and $\Delta(G)$ for the minimum and maximum degree of $G$, respectively. The degree of $v$ in $G$ is denoted $d_G(v)$, while $d_G(u,v)$ denotes the codegree of $u$ and $v$. We omit the subscripts when this is not ambiguous.
	
\subsection{Preliminaries}

We start by recalling a few lemmas from \cite{janzerrainbow} which provide upper bounds for the number of cycles with a pair of ``conflicting" edges or vertices. The first result is Lemma 2.1 in \cite{janzerrainbow}.
	
\begin{lemma} \label{lemma:simple with edges}
	Let $k\geq 2$ be an integer and let $G=(V,E)$ be a graph on $n$ vertices. Let $\sim$ be a symmetric binary relation defined over $E$ such that for every $uv\in E$ and $w\in V$, $w$ has at most $s$ neighbours $z\in V$ which satisfy $uv\sim wz$.
	Then the number of homomorphic $2k$-cycles $(x_1,x_2,\dots,x_{2k})$ in $G$ such that $x_ix_{i+1}\sim x_jx_{j+1}$ for some $i\neq j$ (here and below $x_{2k+1}:=x_1$) is at most $$32k^{3/2}s^{1/2}\Delta(G)^{1/2}n^{\frac{1}{2k}}\hom(C_{2k},G)^{1-\frac{1}{2k}}.$$
\end{lemma}

The next result is {\cite[Lemma 2.4]{janzerrainbow}} in the special case $\ell=0$.

\begin{lemma} \label{lemma:bipartite with vertices}
	Let $k\geq 2$ be an integer and let $G=(V,E)$ be a graph on $n$ vertices. Let $X_1$ and $X_2$ be subsets of $V$. Let $\sim$ be a symmetric binary relation defined over $V$ such that
	\begin{itemize}
		\item for every $u\in V$ and $v\in X_1$, $v$ has at most $\Delta_1$ neighbours $w\in X_2$ and amongst them at most $s_1$ satisfies $u\sim w$, and
		\item for every $u\in V$ and $v\in X_2$, $v$ has at most $\Delta_2$ neighbours $w\in X_1$ and amongst them at most $s_2$ satisfies $u\sim w$.
	\end{itemize}
	Let $M=\max(\Delta_1 s_2,\Delta_2 s_1)$.
	Then the number of homomorphic $2k$-cycles $(x_1,x_2,\dots,x_{2k})\in (X_1\times X_2\times X_1\times \dots \times X_2)\cup (X_2\times X_1\times X_2\times \dots \times X_1)$ in $G$ such that $x_i\sim x_j$ for some $i\neq j$ is at most $$32k^{3/2}M^{1/2}n^{\frac{1}{2k}}\hom(C_{2k},G)^{1-\frac{1}{2k}}.$$
\end{lemma}

The proof of the next lemma is almost verbatim the same as that of {\cite[Lemma 5.4]{janzerrainbow}} and is therefore omitted.

\begin{lemma} \label{lemma:blowupbiregular}
	Let $G$ be an $n$-vertex graph with average degree $d>0$. Then there exist $D_1,D_2\geq \frac{d}{4}$ and a non-empty bipartite subgraph $G'$ in $G$ with parts $X_1$ and $X_2$ such that for every $x\in X_1$, we have $d_{G'}(x)\geq \frac{D_1}{256(\log n)^2}$ and $d_{G}(x)\leq D_1$, and for every $x\in X_2$, we have $d_{G'}(x)\geq \frac{D_2}{256(\log n)^2}$ and $d_{G}(x)\leq D_2$.
\end{lemma}

Finally, we recall \cite[Lemma 4.4]{janzerrainbow}.

\begin{lemma} \label{lemma:homcycles}
	Let $G$ be a bipartite graph with parts $X$ and $Y$ such that $d(x)\geq s$ for every $x\in X$ and $d(y)\geq t$ for every $y\in Y$. Then, for every positive integer $k$,
	$$\hom(C_{2k},G)\geq s^kt^k.$$
\end{lemma}

The next lemma is a variant of Lemma \ref{lemma:bipartite with vertices}.

\begin{lemma} \label{lemma:proportion of neighbours}
	Let $k\geq 2$ be an integer and let $G=(V,E)$ be a non-empty graph on $n$ vertices. Let $\sim$ be a symmetric binary relation defined over $V$ such that for every $u\in V$ and $v\in V$, $v$ has at most $\alpha d(v)$ neighbours $w\in V$ which satisfy $u\sim w$. If $\alpha<(2^{20}k^3(\log n)^4 n^{1/k})^{-1}$, then
	there exists a homomorphic $2k$-cycle $(x_1,x_2,\dots,x_{2k})$ in $G$ such that for all $i\neq j$, we have $x_i\not \sim x_j$.
\end{lemma}

\begin{proof}
    Choose $G'$ according to Lemma \ref{lemma:blowupbiregular}. By Lemma \ref{lemma:homcycles}, $\hom(C_{2k},G')\geq \frac{D_1^k D_2^k}{2^{9k}(\log n)^{4k}}$, so $\hom(C_{2k},G')^{\frac{1}{2k}}\geq 2^{-9/2}(\log n)^{-2}D_1^{1/2} D_2^{1/2}$.
    
    For any $u\in V$ and $v\in X_1$, the number of neighbours $w$ of $v$ in $G$ with $u\sim w$ is at most $\alpha d(v)\leq \alpha D_1$. Similarly, for any $u\in V$ and $v\in X_2$, $v$ has at most $\alpha D_2$ neighbours $w$ with $u\sim w$. Hence we may apply Lemma \ref{lemma:bipartite with vertices} with $\Delta_1=D_1,s_1=\alpha D_1,\Delta_2=D_2,s_2=\alpha D_2$. Then $M=\alpha D_1D_2$ and we obtain that the number of homomorphic $2k$-cycles $(x_1,x_2,\dots,x_{2k})$ in $G'$ such that $x_i\sim x_j$ for some $i\neq j$ is at most $32k^{3/2}\alpha^{1/2}D_1^{1/2}D_2^{1/2}n^{\frac{1}{2k}}\hom(C_{2k},G')^{1-\frac{1}{2k}}$. Since $\alpha<(2^{20}k^3(\log n)^4 n^{1/k})^{-1}$, this is less than $2^{-5}(\log n)^{-2}D_1^{1/2}D_2^{1/2}\hom(C_{2k},G')^{1-\frac{1}{2k}}$. But $\hom(C_{2k},G')^{\frac{1}{2k}}\geq 2^{-9/2}(\log n)^{-2}D_1^{1/2} D_2^{1/2}$, so this is in turn less than $\hom(C_{2k},G')$. It follows that there is a homomorphic $2k$-cycle $(x_1,x_2,\dots,x_{2k})$ in $G'$ such that $x_i\not \sim x_j$ for all $i\neq j$.
\end{proof}

The next two results are well-known lower bounds on the number of homomorphic and genuine copies of $C_{2k}$, respectively.

\begin{lemma}[Sidorenko \cite{Si92}] \label{lemma:sidorenko}
	Let $k\geq 2$ be an integer and let $G$ be a graph on $n$ vertices. Then $$\hom(C_{2k},G)\geq \left(\frac{2|E(G)|}{n}\right)^{2k}.$$
\end{lemma}

\begin{lemma}[Morris--Saxton \cite{MS16}] \label{lemma:supersaturation}
	For any $k\geq 2$, there exist $C=C(k)$ and $c=c(k)>0$ such that any $n$-vertex graph with $e\geq Cn^{1+1/k}$ edges has at least $ce^{2k}/n^{2k}$ copies of $C_{2k}$.
\end{lemma}

As is very common in bipartite Tur\'an problems, we will use the following regularization lemma of Erd\H os and Simonovits. A graph $G$ is called $K$-almost regular if $\Delta(G)\leq K\delta(G)$.

\begin{lemma}[Erd\H os--Simonovits \cite{ES70}] \label{lemmaES}
	Let $0<\alpha<1$ and let $n$ be a positive integer that is sufficiently large as a function of $\alpha$. Let $G$ be a graph on $n$ vertices with $e(G)\geq n^{1+\alpha}$. Then $G$ contains a $K$-almost regular subgraph $G'$ on $m\geq n^{\alpha\frac{1-\alpha}{1+\alpha}}$ vertices such that $e(G')\geq \frac{2}{5}m^{1+\alpha}$ and $K=10\cdot 2^{\frac{1}{\alpha^2}+1}$.
\end{lemma}

For convenience, we prove a slight variant of the above result.

\begin{lemma} \label{lemmaESvariant}
	Let $0<\alpha<1$ and let $n$ be a positive integer that is sufficiently large as a function of $\alpha$. Let $G$ be a graph on $n$ vertices with $e(G)\geq n^{1+\alpha}$. Then $G$ contains a subgraph $G''$ on $m\geq n^{\alpha\frac{1-\alpha}{1+\alpha}}$ vertices such that $e(G'')\geq \frac{1}{3}m^{1+\alpha}$ and $\Delta(G'')\leq Km^{\alpha}$, where $K=10\cdot 2^{\frac{1}{\alpha^2}+1}$.
\end{lemma}

\begin{proof}
    By Lemma \ref{lemmaES}, $G$ has a $K$-almost regular subgraph $G'$ on $m\geq n^{\alpha\frac{1-\alpha}{1+\alpha}}$ vertices such that $e(G')\geq \frac{2}{5}m^{1+\alpha}$. Let $G''$ be a random subgraph of $G'$ obtained by keeping each edge of $G'$ independently with probability $\frac{2}{5}m^{1+\alpha}/e(G')$. Since $n$ is sufficiently large, so is $m$, hence by standard concentration inequalities, the probability that $e(G'')<\frac{1}{3}m^{1+\alpha}$ is at most $1/10$. Moreover, $\Delta(G')\leq K\delta(G')\leq 2Ke(G')/m$, so the expected degree in $G''$ of any $v\in V(G')$ is at most $\frac{4}{5}Km^{\alpha}$. Again, by standard concentration inequalities and the union bound, the probability that there exists $v\in V(G')$ whose degree in $G''$ is more than $Km^{\alpha}$ is at most $1/10$. Thus, with positive probability we have both $e(G'')\geq \frac{1}{3}m^{1+\alpha}$ and $\Delta(G'')\leq Km^{\alpha}$.
\end{proof}

Finally, we prove a simple lemma which shows that we can assume that no edge in our host graph is contained in a large proportion of all $4$-cycles.

\begin{lemma} \label{lemma:clean C4s}
	Let $G$ be a non-empty graph on $n$ vertices. Then $G$ has a spanning subgraph $G'$ with at least $e(G)/2$ edges such that if the number of $4$-cycles in $G'$ is $q$, then every edge in $G'$ belongs to at most $\frac{16\log n}{e(G')}q$ copies of $C_4$ in $G'$.
\end{lemma}

\begin{proof}
	If $G$ has at most three edges, then it has no $4$-cycle and the statement is trivial, so let us assume that $e(G)\geq 4$.
	
	Let $G_0=G$. Now, for any $i\geq 0$, if there is an edge in $G_i$ which is contained in more than a proportion $\frac{16\log n}{e(G)}$ of all $4$-cycles in $G_i$, then choose such an edge $e$ and let $G_{i+1}=G_i-e$. Else, set $G'=G_i$.
	
	It is clear that if the number of $4$-cycles in $G'$ is $q$, then every edge in $G'$ belongs to at most $\frac{16\log n}{e(G)}q\leq \frac{16\log n}{e(G')}q$ copies of $C_4$, so it suffices to show that $e(G')\geq e(G)/2$. We claim that $e(G)-e(G')\leq \lceil e(G)/4\rceil$. Indeed, if $G=G_0$ has $t$ copies of $C_4$, then the number of $4$-cycles in $G_i$ is at most $t(1-\frac{16\log n}{e(G)})^i\leq t\exp(-16(\log n)i/e(G))=tn^{-16i/e(G)}$, which is less than one when $i\geq e(G)/4$. Thus, the process must terminate after at most $\lceil e(G)/4\rceil$ steps, and so $e(G)-e(G')\leq \lceil e(G)/4\rceil$. Since $e(G)\geq 4$, it follows that $e(G')\geq e(G)/2$.
\end{proof}

\subsection{The key definitions and deducing the main result}

As we mentioned in Subsection \ref{sec:overview}, we will find a collection of $8k$-cycles in $G$ with a certain property which is closely related to how many ways a $2k$-matching can be extended to a cycle from the family. The next two definitions describe the kind of families that we will want to find.

\begin{definition} \label{def:good}
	Let $G$ be a graph, let $k$ be a positive integer and let $\beta>0$. We call a set $\mathcal{C}\subset V(G)^{8k}$ \emph{$\beta$-good} if there is a positive real number $s$ such that
	\begin{itemize}
		\item for each $\vx=(x_1,x_2,\dots,x_{8k})\in \mathcal{C}$, the vertices $x_1,x_2,\dots,x_{8k}$ are all distinct and $x_ix_{i+1}\in E(G)$ for all $i\in [8k]$ (here and below, the indices are considered modulo $8k$),
		\item for each $\vy\in \cC$ and $i\in \lbrack 8k\rbrack$, there are at most $s$ choices for $\vx\in \cC$ such that $x_j=y_j$ for every $j\neq i$ and
		\item \sloppy for each $1\leq i\leq 8k$, there are at most $\beta|\cC|/(16ks)$ choices for the vertices $y_1,y_2,\dots,y_i,y_{i+3},y_{i+4},\dots,y_{8k}\in V(G)$ such that there exists $\vx\in \cC$ with $x_j=y_j$ for all $j\not \in \{i+1,i+2\} \mod 8k$. 
	\end{itemize} 
\end{definition}

\begin{definition} \label{def:nice}
	Let $G$ be a graph, let $k$ be a positive integer and let $\beta>0$. We call a set $\mathcal{C}\subset V(G)^{8k}$ \emph{$\beta$-nice} if
	\begin{itemize}
		\item for each $\vx=(x_1,x_2,\dots,x_{8k})\in \mathcal{C}$, the vertices $x_1,x_2,\dots,x_{8k}$ are all distinct and $x_ix_{i+1}\in E(G)$ for all $i\in [8k]$ and
		\item for each $1\leq i\leq 8k$,  $y_1,y_2,\dots,y_i,y_{i+3},y_{i+4},\dots,y_{8k}\in V(G)$ and $u\in V(G)$, amongst the elements $\vx\in \mathcal{C}$ with $x_j=y_j$ for all $j\not \in \{i+1,i+2\}\mod 8k$, at most $\beta$ proportion has $x_{i+1}=u$ or $x_{i+2}=u$.
	\end{itemize} 
\end{definition}

The proof of Theorem \ref{thm:turan of Hkl} consists of the following three steps:
\begin{enumerate}
    \item An $n$-vertex graph with at least $\frac{1}{6}n^{4/3+\eps}$ edges and maximum degree $O(n^{1/3+\eps})$ in which no edge belongs to a large proportion of the $4$-cycles contains an $n^{-\eps/2}$-good collection of $8k$-cycles. \label{Step1}
    \item A $\beta$-good collection of $8k$-cycles contains a $\beta$-nice collection of $8k$-cycles. \label{Step2}
    \item The existence of an $n^{-\eps/2}$-nice collection of $8k$-cycles implies the existence of $H_{k,\ell}$ as a subgraph. \label{Step3}
\end{enumerate}

Step \ref{Step1} is the most complicated of the three steps and relies on the following two lemmas, whose proofs will be given in Subsections \ref{sec:fewcycles} and \ref{sec:manycycles}, respectively. The first one deals with the case when there are few $4$-cycles in the graph, while the second one applies when there are many $4$-cycles.

\begin{lemma} \label{lemma:few4cycles}
	Let $\eps,K>0$ be constants and let $k\geq 1/\eps$. Let $n$ be sufficiently large in terms of $\eps$, $K$ and $k$ and let $G$ be an $n$-vertex graph with $e(G)\geq \frac{1}{6}n^{4/3+\eps}$ and $\Delta(G)\leq Kn^{1/3+\eps}$. Assume that each edge is contained in at most $n^{1/3+2\eps}$ copies of $C_4$. Then there is a non-empty $n^{-\eps/2}$-good set $\mathcal{C}\subset V(G)^{8k}$.
\end{lemma}

\begin{lemma} \label{lemma:many4cycles}
	Let $\eps,K>0$ be constants with $\eps<1/6$ and let $k\geq 1/\eps$. Let $n$ be sufficiently large in terms of $\eps$, $K$ and $k$ and let $G$ be an $n$-vertex graph with $\Delta(G)\leq Kn^{1/3+\eps}$. Assume that $G$ has $q\geq \frac{n^{5/3+3\eps}}{96\log n}$ copies of $C_4$ such that every edge is contained in at most $\frac{96\log n}{n^{4/3+\eps}}q$ copies of $C_4$. Then there is a non-empty $n^{-\eps/2}$-good set $\mathcal{C}\subset V(G)^{8k}$.
\end{lemma}

The next two lemmas perform Step \ref{Step2} and Step \ref{Step3}, respectively.

\begin{lemma} \label{lem:good implies nice}
	Let $G$ be a graph, let $k$ be a positive integer and let $\beta>0$. If $\mathcal{C}_0\subset V(G)^{8k}$ is non-empty and $\beta$-good, then it has a non-empty subset which is $\beta$-nice.
\end{lemma}

\begin{proof}
	Let $s$ be the positive real satisfying the three properties in Definition \ref{def:good}. Define $\mathcal{C}_0\supset \mathcal{C}_1\supset \mathcal{C}_2 \supset\dots$ recursively as follows. Having defined $\mathcal{C}_0,\mathcal{C}_1,\dots,\mathcal{C}_t$, if there are some $\vy\in \mathcal{C}_t$ and $i\in \lbrack 8k\rbrack$ such that the number of $\vx\in \mathcal{C}_t$ with $x_j=y_j$ for all $j\not \in \{i+1,i+2\}$ is less than $2\beta^{-1} s$, then let
	\begin{equation}
	\mathcal{C}_{t+1}=\mathcal{C}_t\setminus \{\vx\in \mathcal{C}_t: x_j=y_j \text{ for all } j\not \in \{i+1,i+2\}\} \label{eqn:removal}
	\end{equation}
	for some such $\vy$ and $i$. Else, terminate the recursion and let $\mathcal{C}=\mathcal{C}_t$.
	
	We claim that $\mathcal{C}$ is non-empty. To see this, choose $r$ such that $\mathcal{C}=\mathcal{C}_r$ and note that $|\mathcal{C}_t\setminus \mathcal{C}_{t+1}|< 2\beta^{-1} s$ for each $t<r$. Furthermore, $r\leq 8k \cdot \beta|\cC_0|/(16ks)$ because there are at most $8k$ ways to choose $i$ in (\ref{eqn:removal}), there are at most $\beta|\cC_0|/(16ks)$ ways to choose the vertices $y_1,y_2,\dots,y_i,y_{i+3},y_{i+4},\dots,y_{8k}$ in (\ref{eqn:removal}) (by the third condition in Definition \ref{def:good}) and any such choices feature at most once in the process. Hence, $|\mathcal{C}_0\setminus \mathcal{C}|< 8k \beta|\cC_0|/(16ks)\cdot 2\beta^{-1} s=|\mathcal{C}_0|$. Thus, $\mathcal{C}\neq \emptyset$.
	
	We now verify that $\mathcal{C}$ is $\beta$-nice. Indeed, for any $1\leq i\leq 8k$ and $y_1,y_2,\dots,y_i,y_{i+3},y_{i+4},\dots,y_{8k}\in V(G)$, if there is any $\vx\in \mathcal{C}$ such that $x_j=y_j$ for all $j\not \in \{i+1,i+2\}$, then there are at least $2\beta^{-1} s$ such $\vx\in \cC$. But by the second condition in Definition \ref{def:good}, for any $u\in V(G)$, there are at most $s$ choices for $\vx\in \cC$ such that $x_j=y_j$ for every $j\not \in \{i+1,i+2\}$ and $x_{i+1}=u$. Similarly, there are at most $s$ choices for $\vx\in \cC$ such that $x_j=y_j$ for every $j\not \in \{i+1,i+2\}$ and $x_{i+2}=u$. Thus, $\mathcal{C}$ is indeed $\beta$-nice.
\end{proof}

\begin{lemma} \label{lemma:nice implies subgraph}
	Let $\delta>0$ be a real number and let $k$ and $\ell\geq 8k/\delta$ be positive integers. Let $n$ be sufficiently large in terms of $\delta$, $k$ and $\ell$ and let $G$ be an $n$-vertex graph which has a non-empty $n^{-\delta}$-nice set $\mathcal{C}\subset V(G)^{8k}$. Then $G$ contains $H_{k,\ell}$ as a subgraph.
\end{lemma}

\begin{proof}
	We define an auxiliary graph $\cG$ as follows. The vertex set of $\cG$ is $V(G)^{4k}$. Moreover, $(y_1,y_2,\dots,y_{4k})$ and $(z_1,z_2,\dots,z_{4k})$ are joined by an edge in $\cG$ if and only if
	\begin{align}
		(&y_1,y_2,z_3,z_4,y_5,y_6,z_7,z_8,\dots,y_{4k-3},y_{4k-2},z_{4k-1},z_{4k}, \nonumber \\
		&y_{4k},y_{4k-1},z_{4k-2},z_{4k-3},y_{4k-4},y_{4k-5},z_{4k-6},z_{4k-7},\dots,y_4,y_3,z_2,z_1)\in \mathcal{C} \label{eqn:first containment}
	\end{align}
	or
	\begin{align}
		(&z_1,z_2,y_3,y_4,z_5,z_6,y_7,y_8,\dots,z_{4k-3},z_{4k-2},y_{4k-1},y_{4k}, \nonumber \\
	&z_{4k},z_{4k-1},y_{4k-2},y_{4k-3},z_{4k-4},z_{4k-5},y_{4k-6},y_{4k-7},\dots,z_4,z_3,y_2,y_1)\in \mathcal{C}. \label{eqn:second containment}
	\end{align}
	
	\noindent \emph{Claim.} For any $\vy\in V(\cG)$, and any $u\in V(G)$, the number of $\vz\in N_{\cG}(\vy)$ for which there is an $i\in \lbrack 4k\rbrack$ with $z_i=u$ is at most $8kn^{-\delta}d_{\cG}(\vy)$.
	
	\medskip
	
	\noindent \emph{Proof of Claim.} By symmetry, it suffices to prove that the number of $\vz\in V(G)^{4k}$ for which (\ref{eqn:second containment}) holds and $z_1=u$ is at most $n^{-\delta}d_{\cG}(\vy)$. By the second condition in Definition \ref{def:nice}, for any fixed $z_3,z_4,\dots,z_{4k}\in V(G)$, at most $n^{-\delta}$ proportion of the pairs $(z_1,z_2)\in V(G)^2$ satisfying (\ref{eqn:second containment}) has $z_1=u$. This implies that at most $n^{-\delta}$ proportion of the $4k$-tuples $(z_1,z_2,\dots,z_{4k})\in V(G)^{4k}$ satisfying (\ref{eqn:second containment}) has $z_1=u$. Hence, there are at most $n^{-\delta}d_{\cG}(\vy)$ tuples $\vz\in V(G)^{4k}$ for which (\ref{eqn:second containment}) holds and $z_1=u$.
	$\Box$
	
	\medskip
	
	Define a binary relation $\sim$ over $V(\cG)$ by setting $\vy\sim \vz$ if and only if $y_i=z_j$ for some $i,j\in \lbrack 4k\rbrack$. The claim above implies (using that each $\vx\in V(\cG)$ has $4k$ coordinates) that for any $\vx,\vy\in V(\cG)$, $\vy$ has at most $32k^2n^{-\delta}d_{\cG}(\vy)$ neighbours $\vz$ with $\vx\sim \vz$. Since $\cG$ has $n^{4k}$ vertices and $32k^{2}n^{-\delta}<(2^{20}\ell^3(\log n^{4k})^4 (n^{4k})^{1/\ell})^{-1}$, Lemma \ref{lemma:proportion of neighbours} implies that $\cG$ contains a homomorphic $2\ell$-cycle $(\vx^1,\vx^2,\dots,\vx^{2\ell})$ with $\vx^i\not \sim \vx^j$ for all $i\neq j$. Then the coordinates of $\vx^1,\dots,\vx^{2\ell}$ together form the vertex set of a copy of $H_{k,\ell}$ in $G$, completing the proof.
\end{proof}

We are now in a position to complete the proof of our main result.

\begin{proof}[Proof of Theorem \ref{thm:turan of Hkl}]
    Let $n$ be sufficiently large in terms of $\eps, k$ and $\ell$. It suffices to prove that if $G$ is an $n$-vertex graph with at least $n^{4/3+\eps}$ edges, then $G$ contains $H_{k,\ell}$ as a subgraph. Let $\alpha=1/3+\eps$. By Lemma \ref{lemmaESvariant}, $G$ has a subgraph $G''$ on $m\geq n^{\alpha\frac{1-\alpha}{1+\alpha}}$ vertices such that $e(G'')\geq \frac{1}{3}m^{1+\alpha}$ and $\Delta(G'')\leq Km^{\alpha}$, where $K=10\cdot 2^{\frac{1}{\alpha^2}+1}$. By Lemma \ref{lemma:clean C4s}, $G''$ has a spanning subgraph $F$ with at least $e(G'')/2\geq \frac{1}{6}m^{4/3+\eps}$ edges such that if the number of $4$-cycles in $F$ is $q$, then every edge in $F$ belongs to at most $\frac{16\log m}{e(F)}q$ copies of $C_4$ in $F$. Since $e(F)\geq \frac{1}{6}m^{4/3+\eps}$, this means that every edge in $F$ is in at most $\frac{96 \log m}{m^{4/3+\eps}}q$ copies of $C_4$ in $F$.
    
    Assume that $q\leq \frac{m^{5/3+3\eps}}{96\log m}$. Then every edge in $F$ belongs to at most $m^{1/3+2\eps}$ copies of $C_4$ in $F$, so Lemma \ref{lemma:few4cycles}, Lemma \ref{lem:good implies nice} and Lemma \ref{lemma:nice implies subgraph} together imply that $F$ contains $H_{k,\ell}$ as a subgraph.
    
    Hence, we may assume that $q>\frac{m^{5/3+3\eps}}{96\log m}$. But then Lemma \ref{lemma:many4cycles}, Lemma \ref{lem:good implies nice} and Lemma~\ref{lemma:nice implies subgraph} together imply that $F$ contains $H_{k,\ell}$ as a subgraph.
\end{proof}

In the remaining two subsections, we prove Lemmas \ref{lemma:few4cycles} and \ref{lemma:many4cycles}.

\subsection{Few $4$-cycles: the proof of Lemma \ref{lemma:few4cycles}} \label{sec:fewcycles}

\begin{lemma} \label{lemma:tuples with small codegree}
	Let $\eps,K>0$ be constants and let $k\geq 2/\eps$. Let $n$ be sufficiently large in terms of $\eps$, $K$ and $k$ and let $G$ be an $n$-vertex graph with $e(G)\geq \frac{1}{6}n^{4/3+\eps}$ and $\Delta(G)\leq Kn^{1/3+\eps}$. Assume that each edge is contained in at most $n^{1/3+2\eps}$ copies of $C_4$. Then $G$ has at least $6^{-2k}n^{(1/3+\eps)2k}$ homomorphic $2k$-cycles $(x_1,x_2,\dots,x_{2k})$ such that $x_1,x_2,\dots,x_{2k}$ are distinct and $d(x_i,x_{i+2})\leq n^{2\eps}$ holds for each $1\leq i\leq 2k$ (as usual, indices are considered modulo $2k$).
\end{lemma}

\begin{proof}
	Let $G=(V,E)$. Define a binary relation $\sim$ over $V$ by taking $v\sim v$ for every $v\in V$ and $u\not \sim v$ for any $u\neq v$. Moreover, define a binary relation $\approx$ over $E$ by taking $e\approx f$ if $e$ and $f$ share precisely one vertex and the other vertices from $e$ and $f$ have codegree more than $n^{2\eps}$.
	
	Let $e\in E$ and $w\in V$. We claim that $w$ has at most $2n^{1/3}$ neighbours $z\in V$ which satisfy $e\approx wz$. Indeed, if $w$ is not a vertex of $e$, then there can be at most two such vertices (the vertices in $e$), so we may assume that $e=uw$ for some $u\in V$. Since $e$ is contained in at most $n^{1/3+2\eps}$ copies of $C_4$, $w$ has at most $2n^{1/3}$ neighbours $z\in V$ with $d(u,z)>n^{2\eps}$, as claimed.
	
	Now by Lemma \ref{lemma:simple with edges} with $s=2n^{1/3}$, the number of homomorphic $2k$-cycles $(x_1,x_2,\dots,x_{2k})$ in $G$ such that $x_ix_{i+1}\approx x_jx_{j+1}$ for some $i\neq j$ is at most $64k^{3/2}n^{1/6}\Delta(G)^{1/2}n^{\frac{1}{2k}}\hom(C_{2k},G)^{1-\frac{1}{2k}}$. Furthermore, by Lemma \ref{lemma:bipartite with vertices} with $X_1=X_2=V$, $s_1=s_2=1$, and $\Delta_1=\Delta_2=\Delta(G)$, the number of homomorphic $2k$-cycles $(x_1,x_2,\dots,x_{2k})$ in $G$ such that $x_i\sim x_j$ for some $i\neq j$ is at most $32k^{3/2}\Delta(G)^{1/2}n^{\frac{1}{2k}}\hom(C_{2k},G)^{1-\frac{1}{2k}}$. Combining these two upper bounds, we conclude that the number of homomorphic $2k$-cycles $(x_1,x_2,\dots,x_{2k})$ in $G$ such that $x_i\sim x_j$ or $x_ix_{i+1}\approx x_jx_{j+1}$ for some $i\neq j$ is at most $96k^{3/2}n^{1/6}\Delta(G)^{1/2}n^{\frac{1}{2k}}\hom(C_{2k},G)^{1-\frac{1}{2k}}$. Since $e(G)\geq \frac{1}{6}n^{4/3+\eps}$, Lemma \ref{lemma:sidorenko} implies that $\hom(C_{2k},G)\geq 3^{-2k} n^{(1/3+\eps)2k}$ and so $\hom(C_{2k},G)^{\frac{1}{2k}}\geq \frac{1}{3}n^{1/3+\eps}$. On the other hand, observe that $96k^{3/2}n^{1/6}\Delta(G)^{1/2}n^{\frac{1}{2k}}\leq 96k^{3/2}n^{1/6}K^{1/2}n^{1/6+\eps/2}n^{\frac{1}{2k}}\leq \frac{1}{6}n^{1/3+\eps}$ since $k\geq 2/\eps$ and $n$ is sufficiently large in terms of $\eps$, $K$ and $k$. Thus, the number of homomorphic $2k$-cycles $(x_1,x_2,\dots,x_{2k})$ in $G$ such that $x_i\sim x_j$ or $x_ix_{i+1}\approx x_jx_{j+1}$ for some $i\neq j$ is at most $\frac{1}{2}\hom(C_{2k},G)$ which means that there are at least $\frac{1}{2}\hom(C_{2k},G)\geq 6^{-2k}n^{(1/3+\eps)2k}$ homomorphic $2k$-cycles $(x_1,x_2,\dots,x_{2k})$ in $G$ such that $x_i\not \sim x_j$ for all $i\neq j$ and $x_ix_{i+1}\not \approx x_{i+1}x_{i+2}$ for all $i$. This completes the proof of the lemma.
\end{proof}

\begin{proof}[Proof of Lemma \ref{lemma:few4cycles}]
	By Lemma \ref{lemma:tuples with small codegree} (applied with $4k$ in place of $k$), there is a set $\mathcal{C}$ of at least $6^{-8k}n^{(1/3+\eps)8k}$ homomorphic $8k$-cycles $(x_1,x_2,\dots,x_{8k})$ in $G$ such that $x_1,x_2,\dots,x_{8k}$ are distinct and $d(x_i,x_{i+2})\leq n^{2\eps}$ holds for each $1\leq i\leq 8k$.
	
	We claim that $\mathcal{C}$ is $\beta$-good for $\beta=n^{-\eps/2}$ with the choice of $s=n^{2\eps}$. Indeed, the first condition in Definition \ref{def:good} is trivial and the second one holds because for each $\vx\in \cC$ and $i\in \lbrack 8k\rbrack$, we have $d(x_i,x_{i+2})\leq n^{2\eps}$, so it remains to verify the third condition. But for any $1\leq i\leq 8k$ and $y_1,y_2,\dots,y_i,y_{i+3},y_{i+4},\dots,y_{8k}\in V(G)$ for which there exists $\vx\in \mathcal{C}$ such that $x_j=y_j$ for all $j\not \in \{i+1,i+2\}$, the vertices $y_{i+3},y_{i+4},\dots,y_{8k},y_1,y_2,\dots,y_i$ form a walk of length $8k-3$. This leaves at most $n\Delta(G)^{8k-3}$ choices for these vertices. But $n\Delta(G)^{8k-3}\leq n(Kn^{1/3+\eps})^{8k-3}=K^{8k-3}n^{(1/3+\eps)8k}n^{-3\eps}\leq \beta|\cC|/(16ks)$, so the proof is complete.
\end{proof}

\subsection{Many $4$-cycles: the proof of Lemma \ref{lemma:many4cycles}} \label{sec:manycycles}

When $G$ has many $4$-cycles, we can use them to find a vertex $v$ and a large collection of $8k$-cycles with the property that every second vertex on such an $8k$-cycle is a neighbour of $v$. We will then argue that one can choose such a collection to be $n^{-\eps/2}$-good. The next lemma shows how we choose $v$.
In what follows, we say that a set $\cD$ of triples is \emph{symmetric} if whenever $(x,y,z)\in \cD$, we also have $(z,y,x)\in \cD$.

\begin{lemma} \label{lemma:find 4cycles}
	Let $\eps,K>0$ be constants with $\eps<1/6$. Let $n$ be sufficiently large in terms of $\eps$ and $K$ and let $G$ be an $n$-vertex graph with $\Delta(G)\leq Kn^{1/3+\eps}$. Assume that $G$ has $q$ copies of $C_4$ such that every edge is contained in at most $\frac{96\log n}{n^{4/3+\eps}}q$ copies of $C_4$. Then there exist $v\in V(G)$, a positive real number $s$ and a symmetric set $\cD\subset V(G)^3$ of size at least $\frac{q}{n\log n}$ such that the following hold.
	
	\begin{enumerate}
		\item For each $(x,y,z)\in \cD$, the vertices $v,x,y,z$ form a $C_4$ (in this order).
		\item For each $(x,y,z)\in \cD$, $d(v,y),d(x,z)\leq s$.
		\item For any fixed $x\in V(G)$, there are at most $\frac{384(\log n)q}{n^{4/3+\eps}s}$ vertices $y\in V(G)$ for which there exists $z\in V(G)$ with $(x,y,z)\in \cD$. \label{cond:extension}
	\end{enumerate}
\end{lemma}

\begin{proof}
	Note that there exist at least $q$ quadruples $(v,x,y,z)\in V(G)^4$ such that $v,x,y,z$ form a $4$-cycle (in this order) and $d(x,z)\leq d(v,y)$. Hence, there is some $v\in V(G)$ for which there is a set $\cD_0$ of at least $q/n$ triples $(x,y,z)\in V(G)^3$ such that $v,x,y,z$ form a $4$-cycle and $d(x,z)\leq d(v,y)$. For any $(x,y,z)\in \cD_0$ there is some integer $1\leq i\leq \lceil \log_2 (Kn^{1/3+\eps})\rceil+1$ such that $2^{i-1}\leq d(v,y)< 2^{i}$. Since $\lceil\log_2 (Kn^{1/3+\eps})\rceil+1\leq \log n$, by the pigeon hole principle there exists some $1\leq i\leq \log n$ such that the number of $(x,y,z)\in \cD_0$ with $2^{i-1}\leq d(v,y)< 2^{i}$ is at least $|\cD_0|/\log n$. Write $\cD'$ for the set of these triples and let $s=2^{i}$. Clearly for any $(x,y,z)\in \cD'$, the vertices $v,x,y,z$ form a $4$-cycle and satisfy $d(x,z)\leq d(v,y)\leq s$. Let $\cD=\{(x,y,z)\in V(G)^3: (x,y,z)\in \cD' \text{ or } (z,y,x)\in \cD'\}$. Now $|\cD|\geq |\cD'|\geq \frac{q}{n\log n}$, $\cD$ is symmetric and satisfies the first two conditions in the lemma. Furthermore, for any $(x,y,z)\in \cD$, we have $d(v,y)\geq \max(s/2,2)$. Therefore, the vertices $v,x,y$ can be extended in at least $s/4$ many ways to a $4$-cycle $vxyz$. Thus, for any $x\in V(G)$, there are at most $\frac{96\log n}{n^{4/3+\eps}}q/(s/4)$ vertices $y\in V(G)$ for which there exists $z\in V(G)$ with $(x,y,z)\in \cD$ (since there are at least $s/4$ copies of $C_4$ containing $vx$ corresponding to each such vertex $y$ and $vx$ is an at most $\frac{96\log n}{n^{4/3+\eps}}q$ copies of $C_4$). This shows that the third condition in the lemma is satisfied and completes the proof.
\end{proof}

\begin{lemma} \label{lemma:find longer cycle}
	Let $\eps,K>0$ be constants with $\eps<1/6$ and let $k\geq 2/\eps$. Let $n$ be sufficiently large in terms of $\eps$, $K$ and $k$. Let $T$ be a set of size at most $Kn^{1/3+\eps}$ and let $R$ be a set of size $n$. Let $\cD\subset T\times R\times T$ be a symmetric set of size at least $n^{2/3+5\eps/2}$ such that for any $x,z\in T$ there are at most $Kn^{1/3+\eps}$ elements $y\in R$ with $(x,y,z)\in \cD$. Then there are at least $|\cD|^{2k}/n^{(1/3+\eps)2k+\eps/8}$ tuples $(x_1,y_1,x_2,y_2,\dots,x_{2k},y_{2k})\in T\times R\times \dots \times T\times R$ of distinct vertices with the property that $(x_i,y_i,x_{i+1})\in \cD$ holds for all $i\in \lbrack 2k\rbrack$.
\end{lemma}

\begin{proof}
	For each $x,z\in T$, write $h(x,z)$ for the number of $y\in R$ such that $(x,y,z)\in \cD$. By the assumption in the lemma, we have $h(x,z)\leq Kn^{1/3+\eps}$. Hence, for each $(x,y,z)\in \cD$, there exists some integer $1\leq i\leq \lceil\log_2(Kn^{1/3+\eps})\rceil+1$ such that $2^{i-1}\leq h(x,z)<2^i$. Since $\lceil\log_2(Kn^{1/3+\eps})\rceil+1\leq \log n$, it follows by the pigeon hole principle that there exists $1\leq s\leq Kn^{1/3+\eps}$ such that for at least $|\cD|/\log n$ triples $(x,y,z)\in \cD$, we have $s\leq h(x,z)<2s$. This implies that there are at least $\frac{|\cD|}{2s\log n}$ pairs $(x,z)\in T^2$ such that $s\leq h(x,z)<2s$. In particular, $\frac{|\cD|}{2s\log n}\leq |T|^2\leq K^2n^{2/3+2\eps}$, so $s\geq 100k$ (by a crude estimate). Define a graph $F$ on vertex set $T$ in which $x$ is a neighbour of $z$ if and only if $x\neq z$ and $h(x,z)\geq s$. Observe that this is well-defined since $h(x,z)=h(z,x)$ by assumption that $\mathcal{D}$ is symmetric. Moreover, the number of edges in $F$ is at least $\frac{|\cD|}{4s\log n}\geq \frac{n^{2/3+5\eps/2}}{4Kn^{1/3+\eps}\log n}=\frac{n^{1/3+3\eps/2}}{4K\log n}\geq C|T|^{1+1/k}$, where $C=C(k)$ is the constant from Lemma \ref{lemma:supersaturation}. Hence, by Lemma \ref{lemma:supersaturation}, there are at least $c|T|^{-2k}(|\cD|/(4s\log n))^{2k}$ copies of $C_{2k}$ in $F$. Note that for any $2k$-cycle $x_1x_2\dots x_{2k}$ in $F$, there are at least $(s/2)^{2k}$ tuples $(x_1,y_1,x_2,y_2,\dots,x_{2k},y_{2k})\in T\times R\times \dots \times T\times R$ of distinct vertices with the property that $(x_i,y_i,x_{i+1})\in \cD$ holds for all $i\in \lbrack 2k\rbrack$. (Indeed, there are at least $s\geq 100k$ choices for each $y_i$ such that $(x_i,y_i,x_{i+1})\in \cD$ and at least $s/2$ of these must be distinct from all previously chosen vertices.)
	
	It follows that there are at least $c|T|^{-2k}(|\cD|/(8\log n))^{2k}$ tuples $(x_1,y_1,x_2,y_2,\dots,x_{2k},y_{2k})\in T\times R\times \dots \times T\times R$ of distinct vertices with the property that $(x_i,y_i,x_{i+1})\in \cD$ holds for all $i\in \lbrack 2k\rbrack$. But $c|T|^{-2k}(|\cD|/(8\log n))^{2k}\geq |\cD|^{2k}/n^{(1/3+\eps)2k+\eps/8}$, so the proof is complete.
\end{proof}

\begin{proof}[Proof of Lemma \ref{lemma:many4cycles}]
	By Lemma \ref{lemma:find 4cycles}, we can choose $v\in V(G)$, a positive real number $s$ and a symmetric set $\cD\subset V(G)^3$ of size at least $\frac{q}{n\log n}$ such that the three properties from that lemma hold. Note that $|\cD|\geq n^{2/3+5\eps/2}$ and $\Delta(G)\leq Kn^{1/3+\eps}$, so we can apply Lemma~\ref{lemma:find longer cycle} with $T=N(v)$, $R=V(G)$ and $2k$ in place of $k$. It follows that there is a set $\mathcal{C}$ of at least $|\cD|^{4k}/n^{(1/3+\eps)4k+\eps/8}$ tuples $(x_1,x_2,\dots,x_{8k})\in V(G)^{8k}$ of distinct vertices such that $(x_{2i-1},x_{2i},x_{2i+1})\in \cD$ for all $i\in \lbrack 4k\rbrack$.
	
	We shall now prove that $\cC$ is $\beta$-good for $\beta=n^{-\eps/2}$. The first condition in Definition \ref{def:good} is clear. For the second one, note that $(x_{2i-1},x_{2i},x_{2i+1})\in \cD$ implies that $x_{2i}\in N(x_{2i-1})\cap N(x_{2i+1})$, $d(x_{2i-1},x_{2i+1})\leq s$, $x_{2i-1}\in N(v)\cap N(x_{2i})$ and $d(v,x_{2i})\leq s$. These together imply the second condition in Definition \ref{def:good}. It remains to verify the third condition. Assume first that $z_1,z_2,\dots,z_{8k-2}\in V(G)$ such that there exists $\vx\in \cC$ with $x_j=z_j$ for all $j\not \in \{8k-1,8k\}$. Then $(z_1,z_2,z_3)\in \cD$, so $z_1\in N(v)$, leaving at most $\Delta(G)$ possibilities for $z_1$. Moreover, for any $1\leq i\leq 4k-2$, the condition $(z_{2i-1},z_{2i},z_{2i+1})\in \cD$ implies that $v,z_{2i-1},z_{2i},z_{2i+1}$ form a $4$-cycle. Since any edge is in at most $\frac{96\log n}{n^{4/3+\eps}}q$ copies of $C_4$, it follows that given $z_1,z_2,\dots,z_{2i-1}$, there are at most $\frac{96\log n}{n^{4/3+\eps}}q$ choices for $z_{2i}$ and $z_{2i+1}$. Finally, $(z_{8k-3},z_{8k-2},x_{8k-1})\in \cD$, so by condition \ref{cond:extension} from Lemma \ref{lemma:find 4cycles}, given $z_{8k-3}$ there are at most $\frac{384(\log n)q}{n^{4/3+\eps}s}$ choices for $z_{8k-2}$. This means that overall there are at most $\Delta(G)\cdot (\frac{96\log n}{n^{4/3+\eps}}q)^{4k-2}\cdot \frac{384(\log n)q}{n^{4/3+\eps}s}=\frac{4\Delta(G)}{s}(\frac{96\log n}{n^{4/3+\eps}}q)^{4k-1}$ possibilities for $z_1,z_2,\dots,z_{8k-2}$. On the other hand,
	$$|\cC|\geq |\cD|^{4k}/n^{(1/3+\eps)4k+\eps/8}\geq n^{-\eps/8}\left(\frac{q}{n^{4/3+\eps}\log n}\right)^{4k}.$$
	Since $\Delta(G)\leq Kn^{1/3+\eps}$ and $q\geq \frac{n^{5/3+3\eps}}{96\log n}$, we have
		$$\frac{4\Delta(G)}{s}\left(\frac{96\log n}{n^{4/3+\eps}}q\right)^{4k-1}\leq \frac{\beta|\cC|}{16ks}.$$
	We have therefore proved that there are at most $\beta|\cC|/(16ks)$ choices for the vertices $z_1,z_2,\dots,z_{8k-2}\in V(G)$ such that there exists $\vx\in \cC$ with $x_j=z_j$ for all $j\not \in \{8k-1,8k\}$. By symmetry, this implies that for any $i\in \lbrack 8k\rbrack$ there are at most $\beta|\cC|/(16ks)$ choices for the vertices $z_1,z_2,\dots,z_i,z_{i+3},z_{i+4},\dots,z_{8k}\in V(G)$ such that there exists $\vx\in \cC$ with $x_j=z_j$ for all $j\not \in \{i+1,i+2\}$. Hence, the third condition in Definition \ref{def:good} is satisfied and $\cC$ is indeed $n^{-\eps/2}$-good.
\end{proof}

\section{Concluding remarks}

In the introduction, we mentioned the conjecture of Erd\H os and Simonovits that for any graph~$H$ of minimum degree at least $r$, there exists $\eps>0$ such that $\ex(n,H)=\Omega(n^{2-1/(r-1)+\eps})$. This conjecture was disproved in \cite{janzerrainbow}, where it was shown that for any even $r\geq 2$ and $\eps>0$ there is an $r$-regular graph $H$ such that $\ex(n,H)=O(n^{2-2/r+\eps})$. Moreover, we conjectured there that the same result should hold for odd values of $r$ too.

\begin{conjecture}[\cite{janzerrainbow}] \label{con:regturan}
	Let $r\geq 3$ be odd and let $\eps>0$. Then there exists an $r$-regular graph~$H$ such that $\ex(n,H)=O(n^{2-2/r+\eps})$.
\end{conjecture}

The present paper proves this conjecture for $r=3$. We remark that the probabilistic deletion method shows that for any graph $H$ with minimum degree at least $r$, there is $\delta>0$ such that $\ex(n,H)=\Omega(n^{2-2/r+\delta})$. This implies that, if true, Conjecture \ref{con:regturan} is tight for every~$r$.

Nearly all bipartite graphs whose Tur\'an number is known (up to a multiplicative constant) are large \emph{rooted powers} of trees (see, e.g., \cite{JMY18,KKL18,CJL19,JQ19,janzerlongerbipsub,JQ19new,JJM20}). For a tree $T$ and a positive integer $s$, let us write $s\ast T$ for the graph obtained by taking $s$ vertex-disjoint copies of $T$ and identifying the $s$ copies of each leaf. We will now show that large rooted powers of trees cannot be used to prove Conjecture \ref{con:regturan} and that in fact they cannot even serve as counterexamples to the above conjecture of Erd\H os and Simonovits for any $r$. A powerful general result of Bukh and Conlon \cite{BC17} implies that for any tree $T$ there exists $s_0$ such that for all $s\geq s_0$, we have $\ex(n,s\ast T)=\Omega\left(n^{2-\frac{v(T)-\ell(T)}{e(T)}}\right)$, where $v(T)$, $e(T)$ and $\ell(T)$ are the number of vertices, edges and leaves in $T$, respectively. If $s\ast T$ has minimum degree at least $r$, then every non-leaf in $T$ has degree at least $r$ in $T$. However, an easy computation shows that in this case $\frac{v(T)-\ell(T)}{e(T)}\leq \frac{\ell(T)-2}{(r-1)(\ell(T)-1)-1}$. When $r\geq 3$, this is less than $\frac{1}{r-1}$, so for any tree $T$, if $s$ is sufficiently large and $s\ast T$ has minimum degree at least $r$, then $\ex(n,s\ast T)=\Omega(n^{2-1/(r-1)+\eps})$ for some $\eps>0$ depending on $T$.

\bibliographystyle{abbrv}
\bibliography{bibliography}

\begin{thebibliography}{10}

\bibitem{AKS03}
N.~Alon, M.~Krivelevich, and B.~Sudakov.
\newblock Tur{\'a}n numbers of bipartite graphs and related {R}amsey-type
  questions.
\newblock {\em Combinatorics, Probability and Computing}, 12:477--494, 2003.

\bibitem{BC17}
B.~Bukh and D.~Conlon.
\newblock Rational exponents in extremal graph theory.
\newblock {\em J. Eur. Math. Soc.}, 20:1747--1757, 2018.

\bibitem{CJL19}
D.~Conlon, O.~Janzer, and J.~Lee.
\newblock More on the extremal number of subdivisions.
\newblock {\em Combinatorica}, to appear.

\bibitem{Erd81new}
P.~{Erd\H os}.
\newblock Problems and results in graph theory.
\newblock {\em The theory and applications of graphs (Kalamazoo, MI, 1980)},
  pages 331--341, 1981.

\bibitem{ESonline}
P.~{Erd\H os} and M.~Simonovits.
\newblock Lower bound for {Tur\'an} number for bipartite non-degenerate graphs.
\newblock
  \url{http://www.math.ucsd.edu/~erdosproblems/erdos/newproblems/TuranNondegenerate.html}.

\bibitem{ESonlinenew}
P.~{Erd\H os} and M.~Simonovits.
\newblock {Tur\'an} number for graphs with subgraphs of minimum degree $>2$.
\newblock
  \url{http://www.math.ucsd.edu/~erdosproblems/erdos/newproblems/TuranDegreeConstraint.html}.

\bibitem{Erd67}
P.~Erd\H{o}s.
\newblock Some recent results on extremal problems in graph theory. {R}esults.
\newblock {\em Theory of Graphs (Internat. Sympos., Rome, 1966)}, pages
  117--123, 1967.

\bibitem{Erd90}
P.~Erd\H{o}s.
\newblock Some of my old and new combinatorial problems.
\newblock {\em Paths, flows, and VLSI-layout (Bonn, 1988), Algorithms Combin.},
  9:35--45, 1990.

\bibitem{ESi66}
P.~Erd\H{o}s and M.~Simonovits.
\newblock A limit theorem in graph theory.
\newblock {\em Studia Sci. Math. Hungar.}, 1:51--57, 1966.

\bibitem{ES46}
P.~Erd\H{o}s and A.~H. Stone.
\newblock On the structure of linear graphs.
\newblock {\em Bull. Amer. Math. Soc.}, 52:1087--1091, 1946.

\bibitem{Erd83}
P.~Erd{\H{o}}s.
\newblock Extremal problems in number theory, combinatorics and geometry.
\newblock In {\em Proc. International Congress of Mathematicians (Warsaw,
  1983)}, 1983.

\bibitem{Erd93}
P.~Erd{\H{o}}s.
\newblock Some of my favorite solved and unsolved problems in graph theory.
\newblock {\em Quaestiones Mathematicae}, 16(3):333--350, 1993.

\bibitem{ES70}
P.~Erd{\H{o}}s and M.~Simonovits.
\newblock Some extremal problems in graph theory.
\newblock In {\em Combinatorial theory and its applications, I}, (Proc.
  Colloq., Balatonf{\"{u}}red, 1969), pages 377--390. North-Holland, Amsterdam,
  1970.

\bibitem{Fu91}
Z.~F{\"u}redi.
\newblock On a {T}ur{\'a}n type problem of {E}rd{\H{o}}s.
\newblock {\em Combinatorica}, 11(1):75--79, 1991.

\bibitem{janzerlongerbipsub}
O.~Janzer.
\newblock The extremal number of the subdivisions of the complete bipartite
  graph.
\newblock {\em SIAM Journal on Discrete Mathematics}, 34(1):241--250, 2020.

\bibitem{janzerrainbow}
O.~Janzer.
\newblock Rainbow {Tur\'an} number of even cycles, repeated patterns and
  blow-ups of cycles.
\newblock {\em Israel Journal of Mathematics}, to appear.

\bibitem{JJM20}
T.~Jiang, Z.~Jiang, and J.~Ma.
\newblock Negligible obstructions and {Tur\'an} exponents.
\newblock {\em arXiv preprint arXiv:2007.02975}, 2020.

\bibitem{JMY18}
T.~Jiang, J.~Ma, and L.~Yepremyan.
\newblock On {T}ur\'an exponents of bipartite graphs.
\newblock {\em Combinatorics, Probability and Computing}, to appear.

\bibitem{JQ19new}
T.~Jiang and Y.~Qiu.
\newblock Many {Tur\'an} exponents via subdivisions.
\newblock {\em arXiv preprint arXiv:1908.02385}, 2019.

\bibitem{JQ19}
T.~Jiang and Y.~Qiu.
\newblock Tur{\'a}n numbers of bipartite subdivisions.
\newblock {\em SIAM Journal on Discrete Mathematics}, 34(1):556--570, 2020.

\bibitem{KKL18}
D.~Y. Kang, J.~Kim, and H.~Liu.
\newblock On the rational tur{\'a}n exponents conjecture.
\newblock {\em Journal of Combinatorial Theory, Series B}, 148:149--172, 2021.

\bibitem{MS16}
R.~Morris and D.~Saxton.
\newblock The number of {$C_{2\ell}$}-free graphs.
\newblock {\em Advances in Mathematics}, 298:534--580, 2016.

\bibitem{Si92}
A.~F. Sidorenko.
\newblock Inequalities for functionals generated by bipartite graphs.
\newblock {\em Discrete Mathematics and Applications}, 2(5):489--504, 1992.

\end{thebibliography}
	
\end{document}